\documentclass[12pt, reqno]{amsart}
\usepackage{amsfonts}
\usepackage{bbm}
\usepackage{amscd,amsfonts}
\usepackage{amssymb, eucal, amsfonts, amsmath, xypic, latexsym, tikz}
\usepackage{pifont}
\usepackage{mathrsfs,color}
\usepackage{amsthm,indentfirst,bm,fancyhdr,dsfont}
\usepackage{graphicx}
\usepackage[all]{xy}
\usepackage[CJKbookmarks=true]{hyperref}

\usepackage{mathrsfs}
\usepackage{amsmath}
\usepackage{amssymb}
\usepackage{hyperref}

\newtheorem{theorem}{Theorem}[section]
\newtheorem{lemma}[theorem]{Lemma}
\newtheorem{prop}[theorem]{Proposition}

\newtheorem{corollary}[theorem]{Corollary}
\theoremstyle{definition}
\newtheorem{conj}[theorem]{Conjecture}
\newtheorem{defn}[theorem]{Definition}

\newtheorem{example}[theorem]{Example}

\newtheorem{remark}[theorem]{Remark}

\newtheorem{question}[theorem]{Question}

\numberwithin{equation}{section}

\def\ggg{\mathfrak{g}}
\def\zzz{\mathfrak{z}}

\def\gl{\mathfrak{gl}}

\def\ggg{\mathfrak{g}}

\def\mmm{\mathfrak{m}}
\def\frakp{\mathfrak{p}}

\def\hhh{\mathfrak{h}}

\def\bbb{\mathfrak{b}}
\def\nnn{\mathfrak{n}}

\def\calh{\mathcal{H}}

\def\calv{\mathcal{V}}

\def\cz{\mathcal {Z}}

\def\bi{\mathbf{i}}

\def\bk{\mathbf{k}}

\def\tsb{\textsf{B}}

\def\bbz{\mathbb{Z}}

\def\bbn{\mathbb{N}}

\def\bbr{{\mathbb{R}}}
\def\bbs{{\mathbb{S}}}

\def\bk{\mathbf{k}}
\def\bfe{\mathbf{e}}
\def\bff{\mathbf{f}}

\def\bo{{\bar 1}}
\def\bz{{\bar 0}}
\def\ev{{\text{ev}}}

\def\Lie{\mathsf{Lie}}
\def\ad{\mathsf{ad}}

\def\GL{\text{GL}}
\def\Hom{\text{Hom}}

\def\id{\mathsf{id}}

\def\calg{\mathcal{G}}

\def\scrm{\mathscr{M}}

{\vskip-\lastskip\medskip
  \noindent
  }%
{\qed\par\medskip
  }

\def\GL{\text{\rm GL}}

\begin{document}

\title[Kac-Weisfeiler Conjecture]
{Irreducible modules of modular Lie superalgebras and super version of the first Kac-Weisfeiler conjecture}

\author{Bin Shu}

\address{School of Mathematical Sciences, Ministry of Education Key Laboratory of Mathematics and Engineering Applications \& Shanghai Key Laboratory of PMMP,  East China Normal University, NO. 500 Dongchuan Rd., Shanghai 200241, China} \email{bshu@math.ecnu.edu.cn}






\thanks{This work is partially supported by the National Natural Science Foundation of China (Grant Nos. 12071136, 11771279 and 12271345), supported in part by Science and Technology Commission of Shanghai Municipality (No. 22DZ2229014).}

\begin{abstract}     {Suppose $\ggg=\ggg_\bz+\ggg_\bo$ is a finite-dimensional restricted Lie superalgebra over an algebraically closed field $\bk$ of characteristic $p>2$.
In this article, we propose a conjecture for maximal dimensions of irreducible modules over the universal enveloping algebra $U(\ggg)$ of $\ggg$, as a super generalization of the celebrated first Kac-Weisfeiler conjecture.
It is demonstrated that the conjecture holds for all basic classical Lie superalgebras and all completely solvable restricted Lie superalgebras. In this process, we investigate irreducible representations of solvable Lie superalgebras.}
	\end{abstract}
\maketitle

\maketitle

\section{Lie superalgebras in characteristic $p$}
Since the works  \cite{BKu, BKl, B, SW, CSW, LS} {\sl{etc}}. on irreducible representations of algebraic supergroups in odd characteristic,   especially  Wang-Zhao's work \cite{WZ1} focusing on irreducible representations of basic classical Lie superalgebras, the study of irreducible representations of finite-dimensional restricted Lie superalgebras in odd characteristic has found  big progress. For instance, see \cite{Zhang, ZhengShu, WZ2, Zhao, Y3, ZengS1, ZengS2, PSh} {\sl etc.}  for determination of irreducible modules of classical Lie superalgebras; see \cite{SZhang1, SZhang2, SY, YS1, YS2,  Y2, Y4, YuL,   LWY, WYL} {\sl etc.} for determination of irreducible modules of Cartan-type Lie superalgebras; and  see \cite{SZhang1,  WZ2, Zhao, Y4, ZengS1, ZengS2, PSh} {\sl etc.} for dimensions or character formulas of irreducible modules.   Nevertheless, their irreducible modules are not well-understood. The purpose of this article is to propose a formulation of maximal dimensions of their irreducible modules. In particular, we thoroughly investigate irreducible representations of finite-dimensional solvable Lie superalgebras.

Throughout the paper, the notions of vector spaces (resp. modules, subalgebras) mean vector superspaces (resp. super-modules, super-subalgebras). For simplicity,
we will often omit the adjunct word ``super."  All vector spaces are defined over $\bk$ which is an algebraically closed field of characteristic $p>2$. For superspace $V=V_\bz+V_\bo$, we will mention the super-dimension of $V$ which means $\underline{\dim} V=(\dim V_\bz|\dim V_\bo)$, in the meanwhile we mention the dimension of $V$ which means $\dim_\bk V:=\dim V_\bz+\dim V_\bo$. As  usual, we denote by $V^*$ the linear dual space of $V$.
Throughout the paper, all Lie (super)algebras are finite-dimensional unless other statements.

\subsection{Restricted Lie superalgebras}\label{restricted}
A Lie superalgebra $\ggg=\ggg_\bz\oplus\ggg_\bo$ is called a restricted one if $\ggg_\bz$ is a restricted Lie algebra and $\ggg_\bo$ is a restricted module of $\ggg_\bz$,     {more precisely, there exists a $p$-mapping $[p]: \ggg_\bz\rightarrow \ggg_\bz$
 satisfying
 \begin{itemize}
 \item[(a)] $(kx)^{[p]} = k^px^{[p]}$ for all $k\in\bk$ and $x\in\ggg_\bz$,
\item[(b)] $[x^{[p]}, y] = (\ad x)^p(y)$ for all $x\in\ggg_\bz$ and $y\in\ggg$,
\item[(c)] $(x + y)^{[p]} = x^{[p]} + y^{[p]} +
\sum_{i=1}^{p-1}
 s_i(x, y)$ for all $x, y\in\ggg_\bz$ where
$(\ad (x\otimes t + y\otimes 1))^{p-1}(x\otimes 1)=\sum_{i=1}^{p-1}is_i(x,y)\otimes t^{i-1}\in \ggg_0\otimes_\bk \bk[t]$. Here $\bk[t]$ denotes the polynomial ring over $\bk$ with indeterminant $t$.
\end{itemize}
 With emphasis on the $p$-mapping $[p]$, we sometimes denote the restricted Lie algebra $\ggg_\bz$ by $(\ggg_\bz,[p])$.
One can refer to \cite[\S{V}.7]{Jac} or \cite[Chapter 2]{SF} for more details on restricted Lie algebras and restricted modules.}

Denote by $\cz(\ggg)$ the center of $U(\ggg)$, i.e. $\cz(\ggg):=\{u\in U(\ggg)\mid \ad x(u)=0\;\; \forall x\in \ggg\}$.
For a restricted Lie superalgebra $\ggg$, the $p$-center $\cz_0$ of $U(\ggg_\bz)$ which is defined to be the subalgebra generated by $\{x^p-x^{[p]}\mid x\in \ggg_{\bar 0}\}$, lies in $\cz$.
  Fix a basis $\{x_1,\cdots,x_s\}$ of $\ggg_\bz$ and a basis $\{y_1,\cdots,y_t\}$ of $\ggg_\bo$.
Set $\xi_{i}=x_{i}^{p}-x_{i}^{[p]}, i=1,\ldots, s$. The $p$-center
$\cz_0$ is a polynomial ring
$\bk[\xi_{1},\ldots,\xi_{s}]$ generated by $\xi_{1},\ldots, \xi_{s}$     {(see for example, \cite[\S2.3]{WZ1}).}

By the PBW theorem, one easily knows that the enveloping
superalgebra $U(\ggg)$ is a free module over $\cz_0$ with basis
\[x_{1}^{a_{1}}\cdots x_{s}^{a_{s}}y_{1}^{b_{1}}\cdots y_{t}^{b_{t}},
 0\leq a_i\leq p-1, \; b_{j}\in\{0,1\}\mbox{ for }i=1,\cdots,s, j=1,\cdots,t \]
    {(see for example, \cite[\S2.3]{WZ1}).}

\subsection{Reduced enveloping algebras of restricted Lie superalgebras}\label{reduced}
Suppose $V$ is an irreducible $U(\ggg)$-module.
By the above argument,  for any $x\in \ggg_\bz$, $x^p-x^{[p]}$ lies in the center of $U(\ggg)$.     {By definition, $x^p-x^{[p]}$ acts on $V$ as an even linear transformation for $x\in \ggg_\bz$.}
Schur's Lemma entails that  each $x^p-x^{[p]}$ for $x\in \ggg_\bz$ acts on $V$ by  scalar $\chi(x)^p$ for some $\chi\in {\ggg_\bz}^*$. Such  $\chi$ is called the $p$-character of  $V$.     {Suppose $\chi\in {\ggg_\bz}^*$ is given, which is naturally regarded in $\ggg^*$ by trivial extension.} Denote by $I_\chi$ the ideal of $U(\ggg)$ generated by the even central elements $x^p-x^{[p]}-\chi(x)^p$ with $x$ running over $\ggg_\bz$.  More generally, we can say that  a $U(\ggg)$-module $M$ is  a $\chi$-reduced module for any given $\chi\in \ggg_\bz^*$ if  for any $x\in \ggg_\bz$, $x^p-x^{[p]}$ acts by the scalar $\chi(x)^p$. All $\chi$-reduced modules for any given $\chi\in\ggg_\bz^*$ constitute a full subcategory of the $U(\ggg)$-module category.
The quotient algebra $U_\chi(\ggg):=U(\ggg)\slash I_\chi$ is called the reduced enveloping superalgebra of $p$-character $\chi$.  Then the $\chi$-reduced module category of $\ggg$ coincides with the $U_\chi(\ggg)$-module category. If $\hhh$ is a restricted Lie subalgebra of $\ggg$, we often make use of  $\chi|_{\hhh_{\bz}}$ for $\chi\in \ggg^*_\bz$ when we consider the $U_\chi(\ggg)$-module category and its  objects  induced from $\hhh$-modules. By abuse of notations, we will simply write $\chi|_{\hhh_{\bz}}$ as $\chi$.

By the PBW theorem, the superalgebra $U_\chi(\ggg)$ has a basis \[x_{1}^{a_{1}}\cdots x_{s}^{a_{s}}y_{1}^{b_{1}}\cdots y_{t}^{b_{t}},
 0\leq a_{i}\leq p-1; b_{j}\in\{0,1\}\mbox{ for }i=1,\cdots s; j=1,\cdots,t,\]
and $\dim U_\chi(\ggg)=p^{\dim \ggg_\bz}2^{\dim \ggg_\bo}$.

 The reduced enveloping algebra corresponding to $\chi=0$ is  $U_0(\ggg)$.  We call it the restricted enveloping algebra of $\ggg$. The modules of $U_0(\ggg)$ are called restricted modules of $\ggg$. The following observation is clear.

\begin{lemma} The dimensions of irreducible modules of a restricted Lie algebra $\ggg$ are not greater  than $p^{\dim \ggg_\bz}2^{\dim \ggg_\bo}$.
\end{lemma}

\subsection{Minimal $p$-envelopes of  finite-dimensional Lie superalgebras}     Any finite-dimensional Lie superalgebra  can be embedded in a finite-dimensional restricted Lie superalgebra (see the appendix \S\ref{sec: appendix}).
Let $\ggg=\ggg_\bz+\ggg_\bo$ be any given Lie superalgbra.  There is a minimal finite-dimensional restricted Lie superalgebra $\ggg_p$ such that  $\ggg_p=(\ggg_\bz)_p+\ggg_\bo$ is a $p$-envelope of $\ggg$, and  $(\ggg_\bz)_p$  a $p$-envelope of $\ggg_\bz$ (see Lemma \ref{lem: app} in the appendix section). Then, one can still show that dimensions of all irreducible modules of $\ggg$ has unified upper-bound, by considering its minimal $p$-envelope. There is a natural question.
\begin{question} What is the maximal dimension for irreducible modules over $\ggg$?
\end{question}

\subsection{} With aim at the above question, the purpose of the present paper is to formulate the maximal irreducible dimensions for finite-dimensional restricted Lie superalgeberas over $\bk$,  as a conjecture (see Conjecture \ref{conj}).  This conjecture is regarded a super version of the plausible first Kac-Weisfeiler conjecture (see Remark \ref{rem: sKW con}(2), or \cite{WK}, \cite{Kac2}). The progress of the work on the first Kac-Weisfeiler conjecture can be learnt from  \cite{PrSk}, \cite{MST}.

The main body of the text is devoted to the verification   of the super first Kac-Weisfeiler conjecture in the case of basic classical Lie superalgebras and complete solvable Lie superalgebras.

\section{Maximal dimensions of irreducible modules for a finite-dimensional restricted Lie superalgebra}

    {
Keep the notations and assumption as above. In particular, $\ggg=\ggg_\bz\oplus\ggg_\bo$ is a finite-dimensional restricted Lie superalgebra over $\bk$.}
For any given $\chi\in \ggg_\bz^*$, consider the bilinear form $\textsf{B}_\chi$ on $\ggg$ with regarding $\chi\in \ggg^*$ by trivial extension
$$\tsb_\chi: \ggg\times \ggg\rightarrow \bk, (X,Y)\mapsto \chi([X,Y]).$$
 Set $\ker(\tsb_\chi)=\{X\in \ggg\mid \tsb_\chi(X,\ggg)=0\}$.
\subsection{} { Generally, $\ggg^*$ can be regarded a $\ggg$-module via defining for $X\in\ggg_{|X|}, f\in \ggg^*_{|f|}$, $X.f:\ggg\rightarrow \bk$ with $(X.f)(Y)=-(-1)^{|X||f|}f{([X,Y])}, \forall Y\in\ggg$. Here $|X|$ and $|f|$ denote the parities of the $\bbz_2$-homogeneous element $X\in\ggg$ and $f\in\ggg^*$, respectively.
   So we can define the centralizer of $\chi$ in $\ggg$, which is denoted by $\zzz^\chi$. By definition,
$$\zzz^\chi=\{X\in\ggg\mid X.\chi=0, \text{ equivalently, } \chi([X,\ggg])=0\}.$$
Then this $\zzz^\chi$ is  exactly equal to  $\ker(\tsb_\chi)$. }
Furthermore, with $\tsb_\chi$ we may define bilinear forms  on the spaces $\tilde\ggg:=\ggg\slash \zzz^\chi$, $\tilde\ggg_\bz:=\ggg_\bz\slash \zzz_\bz^\chi$ and $\tilde\ggg_\bo:=\ggg_\bo\slash \zzz_\bo^\chi$ respectively. By abuse of notations, those bilinear forms are still denoted by $\tsb_\chi$.

In the following arguments, we need some conventions and notations.   Let $\lceil a\rceil$ denote the  greatest  integer lower bound of $a$ for a real number $a\in\bbr$, and $\lfloor a\rfloor$ denote the  least integer upper bound of $a$.

\begin{lemma}\label{lem: non-deg bil forms}
The following statements hold.
\begin{itemize}\label{lem: 2.1}
\item[(1)] The centralizer $\zzz^\chi=\zzz^\chi_\bz+\zzz^\chi_\bo$ is a restricted subalgebra of $\ggg$ if $\ggg$ itself is a restricted Lie superalgebra.
\item[(2)]  $\tsb_\chi$ is a  non-degenerate skew-symmetric bilinear form on $\tilde \ggg_\bz$, and  a  non-degenerate skew-symmetric bilinear form on $\tilde\ggg_\bo$.
 Consequently, $\dim (\ggg_\bz-\zzz_\bz^\chi)$ is even.

\item[(3)] Any maximal isotropic space in $\ggg_\bz$ with respect to $\tsb_\chi$  has dimension $\frac{\dim\ggg_\bz+\dim\zzz_\bz^\chi}{2}$.

 \item[(4)] Any maximal isotropic space   in $\ggg_\bo$ with respect to $\tsb_\chi$ has dimension $\frac{\dim\ggg_\bo+\dim\zzz^\chi_\bo}{2}$ if $\dim\ggg_\bo-\dim\zzz^\chi_\bo$ is even, and has dimension $\frac{\dim\ggg_\bo+\dim\zzz^\chi_\bo-1}{2}$ if  $\dim\ggg_\bo-\dim\zzz^\chi_\bo$ is odd.
\end{itemize}
\end{lemma}
\begin{proof} The parts (1) and (2) directly follows from the definition.
As to (3),  we first note that $\zzz^\chi$ is an isotropic subspace of $\ggg$ with respect to $\tsb_\chi$.  From the part (2) it follows that a maximal isotropic subspace $\tilde V$ of $\tilde\ggg_\bz$
has dimension $\frac{\dim \ggg_\bz-\dim \zzz^\chi_\bz}{2}$. So naturally, the preimage space of $\tilde V$ in $V$ which contains $\zzz^\chi$ is a  maximal isotropic subspace of $\ggg_\bz$. This maximal isotropic subspace has dimension $\frac{\dim\ggg_\bz+\dim\zzz_\bz^\chi}{2}$.

As to  the part (4), from the part (2) again  it follows that a maximal isotropic subspace $\tilde W$ of $\tilde\ggg_\bo$ 
has dimension
$\lceil\frac{\dim\ggg_\bo-\dim\zzz^\chi_\bo}{2}\rceil$.
By the same reason,  the preimage space of $\tilde W$ of $\tilde\ggg_\bo$ which contains $\zzz^\chi$ is a  maximal isotropic subspace of $\ggg_\bo$. Consequently, this maximal isotropic subspace has  dimension $\lceil\frac{\dim\ggg_\bo+\dim\zzz_\bo^\chi}{2}\rceil$.

The proof is completed.
\end{proof}

\begin{remark}\label{rem: 2.2} With the notations  $\lceil a\rceil$ and $\lfloor a\rfloor$ for $a\in\bbr$,  Lemma \ref{lem: 2.1}(4) becomes that the maximal isotropic space with respect to $\tsb_\chi$  in $\ggg_\bo$ has dimension $\lceil\frac{\dim\ggg_\bo+\dim\zzz^\chi_\bo}{2}\rceil$. Set     {
 $$d(\ggg,\chi)=(\frac{\dim\ggg_\bz+\dim\zzz^\chi_\bz}{2}| \lceil\frac{\dim\ggg_\bo+\dim\zzz^\chi_\bo}{2}\rceil).$$
 }
 This $d(\ggg,\chi)$ is the maximal super-dimension of the isotropy subspaces of $\ggg$ with respect to $\tsb_\chi$. Set $i(\ggg,\chi)=\underline{\dim}\ggg-d(\ggg,\chi)$. Then
     {
 $$i(\ggg,\chi)=(\frac{\dim\ggg_\bz-\dim\zzz^\chi_\bz}{2}| \lfloor\frac{\dim\ggg_\bo-\dim\zzz^\chi_\bo}{2}\rfloor).$$
 }
 \end{remark}

 \subsection{The set $D(\ggg,\chi)$  of degraded  subalgebras associated with $\chi$}\label{sec: degraded and derived} We regard $\chi\in\ggg_\bz^*$ as a linear function on $\ggg^*$ by trivial extension.     {Associated with $\chi$, we say that a subalgebra $\hhh$ is degraded if $\underline\dim\hhh=d(\ggg,\chi)$ and $\chi(\hhh^{(1)})=0$. Here and further, $L^{(1)}$ for a Lie (super)algebra $L$ denotes the derived subalgebra of $L$, i.e. $L^{(1)}=[L,L]$. Obviously, such subalgebras contain $\zzz^\chi$ if they exist. In this case, they are further restricted subalgebras whenever $\ggg$ is a restricted Lie superalgebra.}

 Denote by $D(\ggg,\chi)$ the set of  all degraded subalgebras $\hhh$  of $\ggg$.


\subsection{} For the simplicity of arguments, we say that a pair of non-negative integers $(a|b)$ is a super-datum. Call $a$ and $b$ its even entry and odd entry, respectively.   For  $\chi\in\ggg^*_\bz\subset\ggg^*$,
    {we set
\begin{align*}
&b^\chi_0=\dim\ggg_\bz-\dim\zzz^\chi_\bz\cr
 & b^\chi_1=\dim\ggg_\bo-\dim\zzz^\chi_\bo.
\end{align*}
Correspondingly, $i(\ggg,\chi)=(\frac{b^\chi_0}{2}| \lfloor\frac{b^\chi_1}{2}\rfloor)$.}
Also set
$$\mathscr{M}(\ggg) =\max_{\chi\in \ggg_\bz^*}p^{\frac{b^\chi_0}{2}}2^{{\lfloor\frac{b^\chi_1}{2}\rfloor}}.$$

\begin{conj}\label{conj} Let $\ggg$ be a finite-dimensional restricted Lie superalgebra over $\bk$.
The maximal dimension of irreducible $\ggg$-modules  is  $\mathscr{M}(\ggg)$. 
\end{conj}


 \begin{remark} \label{rem: sKW con}
     {(1) Clearly,  by definition $\mathscr{M}(\ggg)$ can be expressed as $p^{\frac{b_0}{2}}2^{\lfloor\frac{b_1}{2}\rfloor}$  for some non-negative integers $b_0$ and $b_1$. In general, such $b_0$ and $b_1$ are not necessarily unique. However, we will see that in many cases,  $b_0=\max_{\chi\in\ggg_\bz^*}b_0^\chi$ and $b_1=\max_{\chi\in\ggg_\bz^*}b^\chi_1$, which are unique.

 (2) The formulation in the above conjecture becomes the first Kac-Weisfeiler conjecture when $\ggg_\bo=0$, i.e. a finite-dimensional restricted Lie algebra $\ggg_\bz$ is regarded a restricted Lie superalgebra with the odd part being zero.}

 (3) This conjecture is a super version of the first Kac-Weisfeiler conjecture on irreducible modules of restricted Lie algebras (see \cite{Kac2})\footnote{There is some counterexample against the first Kac-Weisfeiler conjecture for non-restricted Lie algebras, see \cite{T}}. For the latter, the study has been in a great progress, but the question is still open (see  \cite{PrSk}, \cite{MST}). There are remarkable works (see \cite{Kac2}, \cite{Pre},  \cite{WZ1}) concerning another (the second) Kac-Weisfeiler conjecture on irreducible modules of  Lie algebras of reductive groups in prime characteristic and its super version. Some related progress  can be found in  \cite{GT}, \cite{ WZ2},  \cite{ZengS1}, \cite{ZengS2}.
  \end{remark}

\section{Irreducible modules of basic classical Lie superalgebras}
In this section, we suppose $\ggg$ is a basic classical Lie superalgebra over $\bk$.
Then $\ggg=\ggg_\bz\oplus\ggg_\bo$ with even part being a reductive Lie algebra. As to classical Lie superalgebras of type $P$ and $Q$,     {Conjecture \ref{conj} was very recently confirmed by taking quite different and nontrivial arguments (see \cite{RSYZ}).}

\subsection{Basic classical Lie superalgebras}\label{sec: BCL list} We  list  basic classical Lie superalgebras and their even parts over $\bk$ with the restriction on $p$ (see for example \cite{WZ1}). The restriction on $p$ could be relaxed, but we  always assume this  restriction on $p$ in this section). The most important feature is that each basic classical Lie superalgebra listed below admits a nondegenerate even supersymmetric bilinear form.

\vskip0.3cm
\subsubsection{\label{sec: table}}
\begin{tabular}{ccc}
\hline
 Basic classical Lie superalgebra $\frak{g}$ & $\ggg_\bz$  & characteristic of $\bk$\\
\hline
$\frak{gl}(m|n$) &  $\frak{gl}(m)\oplus \frak{gl}(n)$                &$p>2$            \\
$\frak{sl}(m|n)$ &  $\frak{sl}(m)\oplus \frak{sl}(n)\oplus \bk$    & $p>2, p\nmid (m-n)$   \\
$\frak{osp}(m|n)$ & $\frak{so}(m)\oplus \frak{sp}(n)$                  & $p>2$ \\
$\text{F}(4)$            & $\frak{sl}(2)\oplus \frak{so}(7)$                  & $p>15$  \\
$\text{G}(3)$            & $\frak{sl}(2)\oplus \text{G}_2$                    & $p>15$     \\
$\text{D}(2,1,\alpha)$   & $\frak{sl}(2)\oplus \frak{sl}(2)\oplus \frak{sl}(2)$        & $p>3$   \\
\hline
\end{tabular}

\vskip0.3cm
For  Lie superalgebra $\ggg$ in the list, there is an algebraic supergroup $G$  with $\Lie(G)=\ggg$ satisfying
\begin{itemize}
\item[(1)] $G$ has a purely-even subgroup scheme $G_\ev$ which is an ordinary connected reductive algebraic group with $\Lie(G_\ev)=\ggg_\bz$;
\item[(2)] There is a well-defined action of $G_\ev$ on $\ggg$, giving rise to the adjoint action of $\ggg_\bz$.
    \end{itemize}
    {The above algebraic supergroup are usually called  basic classical supergroups, which can be constructed as  Chevalley supergroups (see \cite{FG1, FG2}). Generally,  for an algebraic supergroup $G$ with $\ggg=\Lie(G)$, $\ggg$ does not determine $G$.
    Instead, the theory of super groups  shows that the pair $(G_\ev, \ggg)$ determines $G$ (see for example, \cite[Chapter 7]{CCF}).   The pair $(G_\ev,\ggg)$ is called a super Harish-Chandra pair ({\sl{ibid}}.).  More precisely, the category of  algebraic supergroups is equivalent to the category of super Harish-Chandra pairs ({\sl{ibid.}}).


 One easily knows that $\ggg=\Lie(G)$ for an algebraic supergroup $G$ is a restricted Lie superalgebra (cf. \cite[Lemma 2.2]{SW} or \cite{SZheng1}).}

\subsection{} Let $\ggg$ be a given basic classical Lie superalgebra.
 We fix a Cartan subalgebra $\hhh$ in $\ggg_\bz$.  Denote by $\Phi$ the root system associate with $\hhh$. Then $\Phi=\Phi_0\cup \Phi_1$ where $\Phi_0$ stands for the set of even roots, and $\Phi_1$ for the set of odd roots.
 Fix a triangular decomposition
$\ggg = \nnn^-\oplus \hhh \oplus \nnn^+$ which is equivalent to say, fix a positive root system $\Phi^+$, or to say fix a simple root system $\Delta$. Here $\nnn^\pm$ stand the Lie subalgebras of positive and negative root vectors, respectively. Furthermore, $\Phi^-=-\Phi^+$, and $\Phi^\pm=\Phi_1^\pm\cup \Phi_0^\pm$.  Moreover, without loss of generality we can assume $\chi(\nnn^+_\bz)=0$ for any $\chi\in \ggg_\bz^*$, up to $G_\ev$-conjugation. Let $\bbb= \hhh\oplus \nnn^+$. Then any $\lambda\in \hhh^*$ defines  a one-dimensional $U_\chi(\hhh)$-module $\bk_\lambda$ as long as $\lambda$ satisfies $\lambda(H)^p-\lambda(H^{[p]}) =\chi(H)^p$ for all  $H\in\hhh$.
Set
$$\Lambda(\chi)=\{\lambda\in \hhh^* \mid \lambda(H)^p-\lambda(H^{[p]}) =\chi(H)^p \;\;\forall  H\in\hhh\}$$
which clearly contains $p^{\dim\hhh}$ elements.

The one-dimensional space  $\bk_\lambda$ can be regarded  a $U_\chi(\bbb)$-module with trivial $\nnn^+$-action, because of $\chi(\nnn^+_\bz)=0$. Then we define an induced module (called  a baby Verma module)
$$Z_\chi(\lambda) := U_\chi(\ggg)\otimes_{U_\chi(\bbb)}\bk_\lambda.$$
Obviously,  $\dim Z_\chi(\lambda)=p^{\dim\nnn_\bz^-} 2^{\dim \nnn_\bo^-}$.
The following fact is clear.

 \begin{lemma}\label{lem: 3.1}
 Any irreducible $U_\chi(\ggg)$-module has dimension not bigger than
$\dim Z_\chi(\lambda)$.
\end{lemma}

    {
The proof is standard. We give an account on it.  By the above arguments, we only need to consider the case $\chi(\nnn^+)=0$. In this case, for any irreducible $U_\chi(\ggg)$-module $V$, $\nnn^+$ acts nilpotently on $V$. Hence $V$ admits one-dimensional $U_\chi(\bbb)$-module  $\bk_\lambda$ with $\nnn^+$-trivial action, and $U_\chi(\hhh)$-action by some function $\lambda\in \Lambda(\chi)$. So $V$ coincides with  $U_\chi(\nnn^-)\bk_\lambda$ which is isomorphic to an irreducible quotient of $Z_\chi(\lambda)$. The lemma follows.}



\subsection{}
Recall that $\ggg_\bz=\Lie(G_\ev)$, and  any element $X\in\ggg_\bz$ admits Jordan-Chevalley decomposition $X=X_s+X_n$ with $X_s$ being  semisimple and $X_n$ nilpotent. Note that for a  basic classical Lie algebra $\ggg$ listed in \ref{sec: table},
there is a non-degenerate $G_\bz$-invariant symmetric bilinear form $(\cdot,\cdot)$ on $\ggg_\bz$. Hence there is a $G$-equivariant isomorphism between $\ggg_\bz$ and $\ggg_\bz^*$.  Consequently, for any $\chi\in\ggg_\bz^*$, there exists a unique $X$ such that $\chi=(X,-)$, which leads to the Jordan-Chevalley decomposition $\chi=\chi_s+\chi_n$ with $\chi_s=(X_s,-)$ and $\chi_n=(X_n,-)$. Furthermore,
  there exists $g\in G_\ev$ such that  $(g.\chi_s)(\nnn_\bz^\pm)=0$, and $(g.\chi_n)(\bbb_\bz)=0$ where the action of $g.$ means the coadjoint action. We say that $\chi$ is  semisimple if   $\chi=\chi_s$, and that $\chi$ is  nilpotent if $\chi=\chi_n$.
For simplicity, we always assume that $\chi=\chi_s+\chi_n$ with $\chi_s(\nnn_\bz^\pm)=0$, and $\chi_n(\bbb_\bz)=0$ in the following because the coadjoint action gives rise to an isomorphism between $U_\chi(\ggg)$ and $U_{g.\chi}(\ggg)$.

     For a semisimple $p$-character $\chi\in \ggg_\bz^*$, we say that $\chi$ is  regular semisimple if  $\chi(H_\alpha)$ are nonzero  for all $\alpha$ from $\Phi$ where $H_\alpha$ is the Cartan toral  element corresponding to $\alpha$. By Lei Zhao's result we have

\begin{theorem} (\cite[Theorems 4.6, 4.7]{Zhao}) \label{thm: zhao}
Suppose $\chi$ is regular semisimple, and $\lambda\in \Lambda_\chi$. Then $Z_\chi(\lambda)$ is irreducible.
\end{theorem}

For a regular semisimple $p$-character $\chi\in\ggg_\bz^*$,  from the definition it follows that  $\zzz^\chi=\hhh$. Hence $(\frac{b_0^\chi}{2}\mid {\lfloor\frac{b_1^\chi}{2}\rfloor})=(\dim \nnn^-_\bz\mid \dim\nnn^-_\bo)$. Correspondingly, $\dim Z_\chi(\lambda)=
p^{\frac{b_0^\chi}{2}}2^{{\lfloor\frac{b_1^\chi}{2}\rfloor}}$, which coincides with $p^{\frac{b_0}{2}}2^{{\lfloor\frac{b_1}{2}\rfloor}}$. On the other hand, Wang-Zhao's theorem concerning Kac-Weisfeiler property (see \cite[Theorem 4.3]{WZ1}) says that for any $\chi\in\ggg_\bz^*$ and any irreducible $U_\chi(\ggg)$-module $V$, $\dim V$ is divisible by  $p^{\frac{b_0^\chi}{2}}2^{{\lfloor\frac{b_1^\chi}{2}\rfloor}}$. Hence we have
$$\dim V\leq \dim Z_\chi(\lambda)= p^{\frac{b_0^\chi}{2}}2^{{\lfloor\frac{b_1^\chi}{2}\rfloor}}.$$
 Combining the above with Lemma \ref{lem: 3.1} and Theorem \ref{thm: zhao}, we finally have the following result.

\begin{corollary} Suppose $\ggg$ is a basic classical Lie superalgebra over $\bk$.     {Then the maximal dimension of irreducible $U(\ggg)$-modules is exactly $\scrm(\ggg)$. Moreover, $\scrm(\ggg)$ can be precisely described as  $p^{\frac{b_0}{2}}2^{{\lfloor\frac{b_1}{2}\rfloor}}$ for $b_0=\max\{b^\chi_0|\chi\in\ggg_\bz^*\}$ and $b_1=\max\{b^\chi_1|\chi\in\ggg_\bz^*\}$.}
\end{corollary}

\subsection{} Consequently, the statement of Conjecture \ref{conj} is true for basic classical Lie superalgebras.

\section{Irreducible modules of solvable Lie superalgebras}

In the next two sections, we will study irreducible representations of finite-dimensional solvable Lie superalgebras over $\bk$, by exploiting the arguments for ordinary solvable Lie algebras (see \cite{WK},  \cite{Sch}, or \cite[Chapter 5]{SF}).     {Keep the notations and assumptions as before. In particular, for a Lie (super)algebra $L$ we denote by $L^{(1)}$ the derived subalgebra of $L$, i.e. $L^{(1)}=[L,L]$.}

\subsection{Basic properties on solvable Lie superalgebras}
The following results are important for the later arguments.

 \begin{lemma}\label{lem: basic lem}  Let $\ggg=\ggg_\bz+\ggg_\bo$ is a finite-dimensional Lie superalgebra over $\bk$.
 \begin{itemize}
 \item[(1)] Suppose $V$ is an irreducible module of $\ggg$. If all elements of $[\ggg,\ggg]$ act nilpotently  on $V$, then $V$ is one-dimensional.
     \item[(2)] Suppose additionally, $\ggg$ is solvable and non-abelian,  then the center $C(\ggg):=\{X\in\ggg\mid \ad(X)(\ggg)=0\}$ does not contain all abelian ideals of $\ggg$.
     \end{itemize}
 \end{lemma}

\begin{proof} (1) It is an obvious fact.

(2)  We prove this statement by reductio ad absurdum. Suppose $C(\ggg)$ contains all abelian ideals. We intend to deduce a contradiction.

 Note that $C(\ggg)$ is an ideal of $\ggg$. Consider the natural surjective homomorphism of Lie superalgebras $\bar{}:\ggg\rightarrow \ggg\slash C(\ggg)$. Of course,  $\bar{\ggg}= \ggg\slash C(\ggg)$.
 By assumption, $\bar \ggg\ne 0$ which is still a solvable Lie superalgebra.
Take a minimal ideal $I=I_\bz+I_\bo\lhd\ggg$ containing $C(\ggg)$ properly. The solvableness of $\ggg$ yields that $I$ properly contains $I^{(1)}$. So $I^{(1)}\subset C(\ggg)$.
Under the assumption of $C(L)$ containing all abelian ideas, it follows that $I^{(1)}\ne 0$. Consequently, $I_\bz\ne 0$.
Moreover, there exists a linear function $\lambda\in C(\ggg)^*$ such that $\lambda|_{I^{(1)}}\ne 0$. Note that $\bar{I}\ne 0$ and it becomes an irreducible $\ggg$-modules. This  irreducible representation of $\ggg$ on $\bar{I}$ is denoted by $(\rho,\bar{I})$.   The linear function  $\lambda$ gives rise to a bilinear form $\Lambda$ on $\bar{I}$ by defining
$\Lambda:\bar I\times \bar I\rightarrow \bk$ via $\Lambda(\bar v_1,\bar v_2)=\lambda([\bar v_1,\bar v_2])$ for any $v_1,v_2\in I$. Then it is easily checked that $\Lambda$ satisfies $\ggg_\bz$-invariant property in the sense that
\begin{align}\label{eq: inv equa}
\Lambda(X.\bar v_1, \bar v_2)+\Lambda(\bar v_1, X.\bar v_2)=0 \text { for }X\in\ggg_\bz, \bar v_i\in \bar I,\; (i=1,2)
\end{align}
 where $X.\bar v_i=\overline{[X,v_i]}$. As $\lambda\ne 0$, we can take $\bar v,\bar w$ such that $\Lambda(\bar v,\bar w)\ne 0$.

We claim  $\rho(\ggg)=0$, which means $\ggg$ acts trivially on $\bar I$. Otherwise, if $\rho(\ggg)$ is nonzero, then there exists a nonzero abelian ideal $J$ because $\rho(\ggg)$ is solvable. For any given nonzero $Z\in J_\bz$, consider the action of $Z^p$ on $\bar I$. Note that for any $Y\in \rho(\ggg)$, $[Z^p,Y]=\ad(Z)^p Y\in J^{(1)}=0$.  By Schur's Lemma, there exists $\alpha(Z)\in \bk$ such that $Z^p.\bar v=\alpha(Z)\bar v$ for any
$Z^p\bar v=\alpha(Z)\bar v$. Comparing with (\ref{eq: inv equa}), we have
$$\alpha(Z)\Lambda(\bar v,\bar w)=\alpha(Z)\Lambda(Z^p.\bar v,\bar w)
=\Lambda(\bar v,Z^p\bar w)=-\alpha(Z)\Lambda(\bar v,\bar w).$$
Hence $\alpha(Z)=0$ for all $Z\in J_\bz$. Thus $J_\bz$ acts nilpotently on $\bar I$, naturally so does $J_\bo$. But $\bar I$ is irreducible, which implies that $J$ acts trivially on $\bar I$. This is to say $J=0$, a contraction.

By the arguments, it is already deduced that $\rho(\ggg)=0$ and the irreducible $\ggg$-module $\bar I$ must be one-dimensional. This means, $I=\bk v+C(\ggg)$ for $v\in\ggg$, which leads to a contradiction $I^{(1)}=0$. The proof is completed.
\end{proof}

\subsection{Induced modules for solvable restricted Lie superalgebras}\label{sec: para4.2} Now we begin to investigate irreducible modules of solvable restricted Lie superalgebras. In the following two subsections,  we assume that $\ggg=\ggg_\bz+\ggg_\bo$ is a finite-dimensional solvable restricted Lie superalgebra, and $\chi\in\ggg_\bz^*$ which will be often regarded a linear function on $\ggg$ by trivial extension. At first,  Lemma \ref{lem: basic lem}(1) implies the following conclusion.

\begin{lemma}\label{lem: moving} Let $\ggg=\ggg_\bz+\ggg_\bo$ be a finite-dimensional solvable restricted Lie superalgebra, and $\chi\in\ggg^*_\bz$.  If $\ggg^{(1)}$  is nilpotent, and  $\chi(\ggg^{(1)})=0$, then any irreducible $U_\chi(\ggg)$-module is one-dimensional.
\end{lemma}

Next, let $\rho:\ggg\rightarrow \gl(V)$ be a finite-dimensional representation of $U_\chi(\ggg)$. Suppose $I$ is an ideal of $\ggg$ such that $I^{(1)}\subset \ker\rho$ and $[\ggg,I]\nsubseteq \ker\rho$.  So there is an irreducible $I$-submodule in $V$ which is one-dimensional. This yields that
\begin{align}\label{eq: V I chi}
V_{I,\chi}:=\{v\in V\mid X.v=\chi(X)v\;\forall X\in I\}
 \end{align}
 is nonzero. Consider $I^\chi:=\{X\in \ggg\mid \chi([X,I])=0\}$. It is easily shown that $I^\chi$ is a restricted subalgebra of $\ggg$  containing  $I$.
Clearly $I^\chi$ is still solvable, and  $V_{I,\chi}$ becomes a module of $U_\chi(I^\chi)$ from which  we will consider an induced module of $U_\chi(\ggg)$.

 We claim that $I^\chi$ is a proper subalgebra of $\ggg$. We will show this by  reductio ad absurdum.
 If  $\ggg=I^\chi$, then $V_{I,\chi}$ is a $U_\chi(\ggg)$-submodule and therefore coincides with $V$, which means $\rho|_{[\ggg,I]}=\chi|_{[\ggg,I]}\id_V=0$. Thus $[\ggg,I]\subset \ker\rho$, which contradicts the condition of $I$.
So we can take a cobasis $\{e_1,\ldots,e_s\}$ of $I^\chi_\bz$ in $\ggg_\bz$ for $s=\dim\ggg_\bz-\dim I^\chi_\bz$, and a cobasis $\{f_1,\ldots,f_t\}$ of $I^\chi_\bo$ in $\ggg_\bo$ for $t=\dim\ggg_\bo-\dim I^\chi_\bo$, which means
\begin{align*}
\ggg_\bz=I_\bz^\chi\oplus \bigoplus_{i=1}^s\bk e_i  \quad\text{ and }\quad
\ggg_\bo=I_\bo^\chi\oplus \bigoplus_{i=1}^t\bk f_i.
\end{align*}

Now we  consider the induced module $\calv:=U_\chi(\ggg)\otimes_{U_\chi(I^\chi)}V_{I,\chi}$ which can be expressed as a vector space
$$\calv=\sum_{(\alpha,\gamma)\in P^s\times E^t}\bk {\bfe}^\alpha{\bff}^\gamma\otimes V_{I,\chi}$$
where $\alpha=(a_1,\ldots,a_s)\in \bbz_{\geq0}^s$ and $\gamma=(c_1,\ldots,c_t)\in \bbz_{\geq0}^t$ are $s$-tuple and  $t$-tuple of non-negative integers, respectively,
  $\bfe^\alpha:=e_1^{a_1}\cdots e_s^{a_s}$, and
${\bff}^\gamma=f_1^{c_1}\cdots f_t^{c_t}$ with
$$P:=\{0,1,\ldots,p-1\}$$ and $$E:=\{0,1\}.$$
Set
$$||(\alpha,\gamma)||:=\sum_i a_i+\sum_j c_j$$ and put
$$\calv_{(l)}=\sum_{\underset{ ||(\alpha,\gamma)||\leq l}{(\alpha,\gamma)\in P^s\times E^t}}\bk \bfe^\alpha\bff^\gamma\otimes V_{I,\chi}.$$

Consider $U_\bz=\sum_{i=1}^s\bk e_i$ and $U_\bo=\sum_{j=1}^t\bk f_j$. Define a linear map $\varphi$ from $U_\bz$ to $I_\bz^*$ by sending any given $X\in I$ onto the function $\tsb_\chi(X,-)$ on $I$. By the definition of $I^\chi$, this $\varphi$ is non-degenerate. Hence there is a set  $\{Z_1,\ldots,Z_s\}\subset I_\bz$ such that $\psi(e_j)(Z_i)=\delta_{ij}$ for $1\leq i,j\leq s$. Here and further, $\delta_{ij}$ denotes the Kronecker function whose value at $(i,j)$ is zero when $i\ne j$ and $1$ when $i=j$.
Similarly, consider a linear map $\psi$ from $U_\bo$ to $I^*_\bo$ sending $Y$ onto the function $\tsb_\chi(Y,-)$ on $I_\bo$, and this $\psi$ is non-degenerate too. We can choose
 $\{T_1,\ldots,T_t\}\subset I_\bo$ such that $\psi(f_j)(T_i)=\delta_{ij}$ for $1\leq i,j\leq t$.

\begin{lemma} Keep the notations and assumptions as above.
For   $\bfe^\alpha\bff^\gamma\otimes v\in \calv_{(l)}$ with $v\in V_{I,\chi}$,  the following formula holds
 \begin{align}\label{eq: degree redu}
 &(Z_i-\chi(Z_i)). \bfe^\alpha\bff^\gamma\otimes v\equiv a_i\bfe^{\alpha-\epsilon_i}\bff^\gamma\otimes v\mod\calv_{(l-2)} \;\text{ if }\alpha\ne 0; \text{ and }\cr
 &T_j. \bff^\gamma\otimes v\equiv c_j\bff^{\gamma-\epsilon_j}\otimes v\mod\calv_{(l-2)} \;\text{ if }\alpha= 0.
\end{align}
\end{lemma}
\begin{proof} We first generally introduce  an order relation $\preceq$ on $\bbz_{\geq0}^q$ (the set of $q$-tuples of non-negative integers) by defining $\kappa\preceq \alpha$ for any given $\kappa'=(k'_1,\ldots,k'_q)$, $\alpha'=(a_1',\ldots,a'_q)\in \bbz_{\geq0}^q$  if and only if $k'_i\leq a'_i$ for all $i$.

Now we turn to the situation  for the lemma.
Set $\epsilon_k:=(\delta_{1k},\ldots, \delta_{sk})$, $k=1,\ldots, s$.
 Recall the following formula for $Z\in I_\bz$
$$Z\bfe^\alpha=\sum_{0\preceq\kappa\preceq\alpha}(-1)^{||\kappa||}
{\alpha\choose\kappa}\bfe^{\alpha-\kappa}\ad(e_s)^{k_s}(\cdots( \ad(e_1)^{k_1}))Z,$$
here and further  we set
 $$||\kappa||:=\sum_{i=1}^sk_i$$ for $\kappa=(k_1,\ldots,k_s)\in\bbz_{\geq0}^s$, and set
 $${\alpha\choose\kappa}:=\sum_{i=1}^s{a_i\choose k_i}$$
and for $T\in I_\bo$
$$T\bff^\gamma=\sum_{0\preceq\kappa\preceq\gamma}\bff^{\gamma-\kappa}\ad(f_t)^{k_t}
(\cdots (\ad(f_1)^{k_1}))T.$$

So in the case when  $\alpha\ne 0$, we have
$$(Z_i-\chi(Z_i)).\bfe^\alpha\bff^\gamma\otimes v
=\sum_{0\preceq\kappa\preceq\alpha}(-1)^{||\kappa||}
{\alpha\choose\kappa}\bfe^{\alpha-\kappa}\ad(e_s)^{k_s}(\cdots (\ad(e_1)^{k_1}))(Z_i-\chi(Z_i)) \bff^\gamma\otimes v.$$
Note that any elements from $I^\chi_\bo$  act trivially on $v$. So we can write
\begin{align*}
(Z_i-\chi(Z_i).\bfe^\alpha\bff^\gamma\otimes v
&\equiv \bfe^\alpha\bff^\gamma(Z_i-\chi(Z_i))\otimes v-
\sum_{k=1}^s a_k\bfe^{\alpha-\epsilon_k}\bff^\gamma [e_k,Z_i]\otimes v
\mod \calv_{(l-2)}\cr
&\equiv a_i\bfe^{\alpha-\epsilon_i}\bff^\gamma \otimes v
\mod \calv_{(l-2)}.
\end{align*}
By the same arguments, we can deal with the case when $\alpha=0$. We finally  have
 \begin{align*}
 T_j\bff^\gamma\otimes v
 &\equiv
\bff^\gamma T_j\otimes v+
 \sum_{k=1}^t \bff^{\gamma-\epsilon_k}[f_k, T_j]\otimes v\mod \calv_{(l-2)}\cr
&\equiv \bff^{\gamma-\epsilon_j}\otimes v
\mod \calv_{(l-2)}.
\end{align*}
The proof is completed.
\end{proof}

\begin{lemma} \label{lem: key lemma}
Keep the above notations and assumptions. In particular, let $\ggg$ be a solvable restricted Lie superalgebra. Let $\rho: \ggg\rightarrow \gl(V)$ be an irreducible representation of $U_\chi(\ggg)$ on $V$.
The following statements hold.
\begin{itemize}

\item[(1)] The module  $V_{I,\chi}$  defined in (\ref{eq: V I chi})
     is an irreducible $U_{\chi}(I^\chi)$-module.

\item[(2)] Furthermore, if we set $\calv:=U_\chi(\ggg)\otimes_{U_\chi(I^\chi)}V_{I,\chi}$, then $V\cong\calv$.
\end{itemize}
\end{lemma}
\begin{proof}
(1) It is clear.

(2) Suppose $W$ is a nonzero submodule of $U_\chi(\ggg)$ in $\calv$. Set $W_0:=\{w\in V_{I,\chi}\mid 1\otimes w\in W\}$.  Obviously,  $W_0$ is a $U_\chi(I^\chi)$-submodule of $V_{I,\chi}$. According to Part (1), $V_{I,\chi}$ is an irreducible $U_\chi(I^\chi)$-module.  Hence $W_0=V_{I,\chi}$ or $W_0=0$. If the former case occurs, then $W=\calv$. Hence it suffices to show that $W_0$ is nonzero.

Recall
$\calv=\sum_{(\alpha,\gamma)\in P^s\times E^t}\bk {\bfe}^\alpha{\bff}^\gamma\otimes V_{I,\chi}$.
Obviously,
$ W=\bigcup_{l\geq0} (W\cap \calv_{(l)})$.
On the other hand, we put $W_{(l)}:=\sum_{\underset{ ||(\alpha,\gamma)||\leq l}{(\alpha,\gamma)\in P^s\times E^t}}\bk \bfe^\alpha\bff^\gamma\otimes W_0$.
We claim that
$$W=\sum_{(\alpha, \gamma)\in P^s\times E^t}\bk {\bfe}^\alpha{\bff}^\gamma\otimes W_0.$$
It  yields that $W_0\ne 0$, which is our purpose.
In order to verify the claim, it suffices to show that
\begin{itemize}
\item[(*)] $W_{(l)}=\calv_{(l)}\cap W$ for all $l\geq 0$.
\end{itemize}
By definition, (*) is true for $l=0$. We now prove that $W\cap \calv_{(l)}\subset W_{(l)}$ by induction on $l$. Let $l\geq 1$ and assume that $W\cap \calv_{(l-1)}\subset W_{(l-1)}$.  Suppose that $v\in W\cap \calv_{(l)}$ is arbitrarily given, we intend to show that $w\in W_{(l)}$. Take a cobasis $\{v_1,\ldots, v_q\}$ of $W_0$ in $V_{I,\chi}$. Without loss of generality, we might as well assume $$v=\sum_{k=1}^q\sum_{\underset{ ||(\alpha,\gamma)||\leq l}{(\alpha,\gamma)\in P^s\times E^t}}C_{\alpha,\gamma,k}\bfe^\alpha \bff^\gamma\otimes v_k$$
with all $C_{\alpha,\gamma,k}\in \bk$.
By (\ref{eq: degree redu}), we have for $i=1,\ldots,s,$
$$(Z_i-\chi(Z_i)).v\equiv \sum_{k}\sum_{||(\alpha,\gamma)||=l}C_{\alpha,\gamma,k}\bfe^{\alpha-\epsilon_i} \bff^\gamma\otimes v_k\mod\calv_{(l-2)}.$$
Thus $(Z_i-\chi(Z_i)).v\in W\cap \calv_{(l-1)}$. By the inductive hypothesis, it lies in $W_{(l-1)}$. Hence all $C_{\alpha,\gamma,k}=0$ as long as $||(\alpha,\gamma)||=l$ and $\alpha\ne 0$.  In the same way, let us deal with the case when $|(\alpha,\gamma)|=l$ and $\alpha= 0$. Multiplication by $T_j$ on $v$, $j=1,\ldots,t$ yields
$$T_j.v\equiv \sum_{k}\sum_{||(\alpha,\gamma)||=l}C_{0,\gamma,k} \bff^{\gamma-\epsilon_j}\otimes v_k\mod\calv_{(l-2)}.$$
 For the same reason as the previous case, we have that $C_{0,\gamma,k}=0$ for all $\gamma$ with $||(0,\gamma)||=l$. Hence $v$ must be zero. In summary,  $W\cap \calv_{(l)}\subset W_{(l)}$. Finally, we have $W\cap \calv_{(l)}=W_{(l)}$. This implies $W=\calv$. Consequently, $\calv$ is an irreducible $U_\chi(\ggg)$-module. On the other hand, there is a natural surjective $U_\chi(\ggg)$-homomorphism from $\calv$ onto $V$. The irreducibility of both $V$ and $\calv$ implies that the natural surjective homomorphism must be an isomorphism.
\end{proof}

\subsection{Irreducible modules of solvable restricted Lie superalgebras} Furthermore, we have the following result.
\begin{prop}\label{prop: 4.5} Keep the assumptions and notations as above. Suppose $\rho:\ggg\rightarrow \gl(V)$  is a finite-dimensional irreducible representation of $U_\chi(\ggg)$. Then there is a subalgebra $\hhh=\hhh_\bz+\hhh_\bo$ such that
\begin{itemize}
\item[(i)]
$\dim V\geq p^{\dim\ggg_\bz-\dim\hhh_\bz}2^{\dim\ggg_\bo-\dim\hhh_\bo}$, and
\item[(ii)] $V$ contains a one-dimensional $\hhh$-submodule.
\end{itemize}
\end{prop}
\begin{proof}
 Consider $\ker\rho$ (the kernel of $\rho$) which is an ideal of $\ggg$.
If $\ggg\slash \ker\rho$ is abelian, then $[\ggg,\ggg]$ acts trivially on $V$.
Hence $V$ is certainly one-dimensional with trivial action of $[\ggg_\bz,\ggg_\bz]+\ggg_\bo$. So the proposition is true in this case while we take $\hhh$ to be $\ggg$ itself.

 In the following, we suppose that $\overline\ggg:=\ggg\slash \ker\rho$ is not abelian. We will prove the proposition by induction on $\dim\ggg$ by steps.

 (1) Keep in mind the assumption that $\overline\ggg:=\ggg\slash \ker\rho$ is not abelian. So the center $C(\overline\ggg)$ is a proper subalgebra of $\overline\ggg$. So $C(\overline\ggg)$ does not contain all abelian ideal of $\ggg$, due to Lemma \ref{lem: basic lem}(2). So, there exists an ideal $I$ of $\ggg$ such that $I^{(1)}\subset \ker\rho$ and $[\ggg,I]\nsubseteq \ker\rho$.
  Note that $I^{(1)}$ acts trivially on $V$. So $\chi(I^{(1)})=0$. Hence there is an irreducible $I$-submodule in $V$ which is one-dimensional. Still denote $V_{I,\chi}=\{v\in V\mid X.v=\chi(X)v\;\forall X\in I\}$. Then we have $V_{I,\chi}$ is nonzero. Keep the notation $I^\chi=\{X\in \ggg\mid \chi([X,I])=0\}$. It is already known that $I^\chi$ is a solvable restricted subalgebra of $\ggg$. By the arguments in the first paragraph of \S\ref{sec: para4.2},
 $I^\chi$ is a proper subalgebra of $\ggg$.
   By Lemma \ref{lem: key lemma},  $V_{I,\chi}$ is an irreducible $U_\chi(I^\chi)$-module, and
 \begin{align*}
 V\cong U_\chi(\ggg)\otimes_{U_\chi(I^\chi)}V_{I,\chi}.
 \end{align*}
 Correspondingly,
 \begin{align}\label{eq: basic iso}
 \dim V=p^{\dim \ggg_\bz-\dim I^\chi_\bz}2^{\dim \ggg_\bo-\dim I^\chi_\bo}\dim V_{I,\chi}.
 \end{align}

(2) Note that $I^\chi$ already turns out to be a proper solvable restricted subalgebra. By the inductive hypothesis, there exits a solvable restricted subalgebra $\hhh$ in $I^\chi$ such that the requirements (i) and (ii) are satisfied with respect to the irreducible $U_\chi(I^\chi)$-module $V_{I,\chi}$.  Hence, by (\ref{eq: basic iso}) we finally have
$\dim V\geq p^{\dim\ggg_\bz-\dim\hhh_\bz}2^{\dim\ggg_\bo-\dim\hhh_\bo}$.
\end{proof}

Combining Proposition \ref{prop: 4.5} and Lemma \ref{lem: key lemma}(2), we conclude the following result.

\begin{corollary}\label{coro: 4.6} Let $\ggg$ be a solvable restricted Lie superalgebra. Then any irreducible module of $\ggg$ must be isomorphic to $U_\chi(\ggg)\otimes_{U_\chi(\hhh)}S$ for  some restricted subalgebra $\hhh$ with $\chi(\hhh^{(1)})=0$ and for a one-dimensional $U_\chi(\hhh)$-module $S$. Correspondingly, any irreducible module has dimension $p^m2^n$ for some $m,n\in\bbn$.
\end{corollary}

\subsection{Irreducible modules of  general finite-dimensional solvable Lie superalgebras which are not necessarily restricted}\label{sec: p-envel and gen solv}

 In this section, we will extend the above results for  solvable restricted Lie subalgebras to any solvable ones.
   Let $\ggg=\ggg_\bz+\ggg_\bo$ be a given solvable Lie superalgbra. {Recall that $\ggg$ admits a minimal finite-dimensional $p$-envelope $\ggg_p$, such that $\ggg_p=(\ggg_\bz)_p\oplus \ggg_\bo$ and $(\ggg_\bz)_p$ is a $p$-envelope of $\ggg_\bz$ (see Lemma \ref{lem: app} in the Appendix).
   This $\ggg_p$  becomes a solvable restricted Lie superalgebra}. As mentioned in Appendix, $\ggg$ is an ideal of $\ggg_p$. More precisely, $\ad(\ggg_\bz)_p(\ggg_\bz)\subset \ggg_\bz^{(1)}$, and $\ad(\ggg_\bz)_p\ggg_\bo\subset \ggg^{(1)}$.

   For irreducible $\ggg$-module $(V,\rho)$, $\rho$ can extend to the one over $\ggg_p$ (see for example \cite[2.5.3]{SF}). Hence $V$ becomes an irreducible $\ggg_p$-modules. Hence there exits a unique $\Upsilon\in (\ggg_\bz)_p^*$,  with which the irreducible $\ggg_p$-module $V$ is associated.  Set $\chi=\Upsilon|_{\ggg_\bz}$.  Then the irreducible $\ggg$-module $V$ is associated with a unique linear function $\chi\in\ggg_{\bz}^*$.

\begin{theorem}\label{lem: general solv} Let $\ggg$ be a solvable Lie superalgebra.
  Any irreducible module $V$ of $\ggg$ is associated with some $\chi\in \ggg_\bz^*$, which has dimension $p^{n}2^{\dim{\ggg_\bo\slash \hhh_\bo}}$ where  $\hhh$ is a subalgebra with $\chi(\hhh^{(1)})=0$ and $V$ contains a one-dimensional $\hhh$-module, and $n$ depends on $\hhh$.
\end{theorem}
\begin{proof} For any given irreducible $\ggg$-module $(V,\rho)$, as arguments above  $V$ becomes an irreducible $\ggg_p$-modules associated with $\Upsilon\in (\ggg_\bz)_p^*$, and $V$ is associated with $\chi:=\Upsilon_{\ggg_\bz}\in \ggg^*_\bz$.
By Corollary \ref{coro: 4.6}, there exists a restricted subalgebra $\frak{H}$ of $\ggg_p$ with $\Upsilon(\frak{H}^{(1)})=\chi(\frak{H}^{(1)})=0$ such that
$V\cong U_\chi(\ggg_p)\otimes_{U_\chi(\frak{H})}S$
where $S\subset V$ is a one-dimensional $\frak{H}$-module.  Correspondingly, $\dim V=p^{\dim{(\ggg_\bz)_p\slash \frak{H}_\bz}}2^{\dim{\ggg_\bo\slash \frak{H}_\bo}}$.

Take  $\hhh=\frak{H}\cap \ggg$. By definition, $\hhh_\bz=\frak{H}_\bz\cap \ggg_\bz$, and $\hhh_\bo=\frak{H}_\bo\cap \ggg_\bo=\frak{H}_\bo$.
 Then $\chi(\hhh^{(1)})=0$, and $\hhh$ has one-dimensional module $S$. 
 Set $n=\dim{(\ggg_\bz)_p\slash \frak{H}_\bz}$. Then
we have  $$\dim V=p^{n}2^{\dim{\ggg_\bo\slash \hhh_\bo}}.$$
The proof is completed.
\end{proof}

\begin{remark}
 (1) The above theorem is an extension of the counterpart result on solvable Lie algebras (see \cite{Sch} or \cite[\S 5.8]{SF}).

 (2) With the above theorem, we can propose the possibility that  super KW property raised by Wang-Zhao in \cite{WZ1} is satisfied with all finite-dimensional solvable Lie superalgebras over $\bk$.
\end{remark}

\section{Irreducible modules of completely solvable Lie superalgebras}

    {A Lie superalgebra $\ggg$ is called completely solvable if $\ggg^{(1)}$ is nilpotent.  Obviously, a completely solvable $\ggg$ is solvable and its even part $\ggg_\bz$ is a completely solvable Lie algebra.}

\subsection{} The following facts are very important to the subsequent arguments.
\begin{lemma}\label{lem: 5.1} Let $\ggg=\ggg_\bz+\ggg_\bo$ be a completely solvable Lie superalgebra.
The following statements hold.
\begin{itemize}
\item[(1)] Any minimal ideal of $\ggg$ is one-dimensional.
\item[(2)]  There exists a sequence of ideals
$\ggg=\ggg_0 \supset \ggg_1\supset \cdots \supset \ggg_{n-1}\supset\ggg_{n}=0$
such that $\dim_\bk \ggg_i=n-i$.
\item[(3)] For $\chi\in\ggg_{\bz}^*$, if $\hhh$ is a subalgebra of codimension one which contains $\zzz^\chi$, then $D(\hhh,\chi)\subset D(\ggg,\chi)$.
    \item[(4)] Each proper subalgebra of $\ggg$ is contained in a subalgebra of codimension one of $\ggg$.
\end{itemize}
\end{lemma}

\begin{proof}  By definition $\ggg^{(1)}$ is nilpotent, it acts nilpotently on $\ggg$ under $\ad$-action. By Lemma \ref{lem: basic lem}(1), The part (1) follows.

As to the part (2),  by (1) there must be minimal ideals of $\ggg$ which is  $\bbz_2$-homogeneous and one-dimensional.
 Notice that any subalgebras and quotients of $\ggg$ are also completely solvable Lie superalgebras. By induction on dimension, the statement follows.

As to (3), we are given $\mmm\in D(\hhh,\chi)$, intending to show $\mmm\in D(\ggg,\chi)$. We first note that   $\underline{\dim}\mmm=d(\hhh,\chi)$.
As $\hhh$ has codimension one, either $\hhh_\bz$ has codimension one in $\ggg_\bz$ while $\hhh_\bo=\ggg_\bo$, or $\hhh_\bo$ has codimension one in $\ggg_\bo$ while $\hhh_\bz=\ggg_\bz$.
For the first case,  $\dim\hhh_\bz=\dim\ggg_\bz-1$. We have that the even entry of the super-datum $d(\hhh,\chi)$ is not less than $\frac{\dim\ggg_\bz-1+\dim \zzz_\bz^\chi}{2}=\frac{\dim\ggg_\bz+\dim \zzz_\bz^\chi}{2}-\frac{1}{2}$, therefore both are equal because both are integers and $\frac{\dim\ggg_\bz+\dim \zzz_\bz^\chi}{2}$ is already an integer. The odd entries  of $d(\hhh,\chi)$ and of $d(\ggg,\chi)$ coincide. So in this case, $d(\hhh,\chi)=d(\ggg,\chi)$.

For the second case,  $\dim\hhh_\bo=\dim\ggg_\bo-1$ while $\hhh_\bz=\ggg_\bz$.
We can  show by similar arguments as in the first case, that $d(\hhh,\chi)=d(\ggg,\chi)$ when $\frac{\dim\ggg_\bo+\dim \zzz_\bo^\chi}{2}$ is an integer.
Suppose $\frac{\dim\ggg_\bo+\dim \zzz_\bo^\chi}{2}$ is not an integer.
Then the odd entries of $d(\hhh, \chi)$ is equal to $\frac{\dim\hhh_\bo+\dim \zzz_\bo^\chi}{2}=\frac{\dim\ggg_\bo-1+\dim \zzz_\bo^\chi}{2}$, which is exactly  $\lceil  \frac{\dim\ggg_\bo+\dim \zzz_\bo^\chi}{2}\rceil$, equal to the odd entry of $d(\ggg,\chi)$. And the even entries of $d(\hhh,\chi)$ and of $d(\ggg,\chi)$ coincide already. Hence  $d(\hhh,\chi)=d(\ggg,\chi)$ in this case. Hence,  we always have $D(\hhh,\chi)\subset D(\ggg,\chi)$ in any case.
The proof of (3) is completed.

Now we prove (4) by induction on dimension.
Suppose $\hhh$ is a proper subalgebra of $\ggg$, and has codimension greater than one. Note that $\ggg_{n-1}$ is an ideal of dimension one. So $\hhh+\ggg_{n-1}$ must be a proper subalgebra of $\ggg$. If $\hhh+\ggg_{n-1}$ has codimension one, then we are done. If it has codimension greater than one, we consider $\phi: \ggg\rightarrow \overline\ggg:=\ggg\slash\ggg_{n-1}$. Then $\phi (\hhh+\ggg_{n-1})$ has codimension greater than one in $\overline\ggg$. Now $\overline{\ggg}$ has dimension less than $\dim_\bk \ggg$.
 By the inductive hypothesis, $\phi(\hhh+\ggg_{n-1})$ is contained in a subalgebra say $\frakp$, of codimension one in $\overline{\ggg}$.
We take  $\phi^{-1}(\frakp)$ the preimage of $\frakp$ in $\ggg$. Then $\phi^{-1}(\frakp)$ is a subalgebra of codimension one in $\ggg$, containing $\hhh$. This subalgebra is desired.
\end{proof}

\begin{corollary}\label{coro: 5.2}
For a completely solvable Lie superalgebra $\ggg$, the set $D(\ggg,\chi)$ is not empty.
\end{corollary}

\begin{proof} If $\zzz^\chi=\ggg$, there is nothing to do because  $\ggg$ itself belongs to $D(\ggg,\chi)$.  So we only need to consider the situation where $\zzz^\chi$ is a proper subalgebra of $\ggg$ in the following.

 By Lemma \ref{lem: 5.1}, it is easily deduced that there are minimal subalgebras in $\ggg$ containing $\zzz^\chi$  and having super-dimension not less than $d(\ggg,\chi)$. We take one, say $\frakp$.
By Lemma \ref{lem: 5.1} again, it is deduced that $D(\frakp,\chi)\subset D(\ggg,\chi)$.
We will show that this $\frakp$ exactly belongs to $D(\frakp,\chi)$, therefore belongs to $D(\ggg,\chi)$.

(1) By applying  Lemma \ref{lem: 5.1}(4) to the completely solvable superalgebra $\frakp$, it is concluded that the proper subalgebra $\zzz^\chi$ is contained in a subalgebra of codimension one in $\frakp$. Hence, the minimality of $\frakp$ yields that $\underline{\dim}\frakp=d(\ggg,\chi)$.

(2) By the above arguments, in $\frakp_\bz$ (resp. $\frakp_\bo$) the maximal isotropic space has  dimension equal to $\dim\frakp_\bz$ (resp. $\dim\frakp_\bo$). Hence $\frakp$ itself is isotropic, which means
$\chi(\frakp^{(1)})=0$. Hence $\frakp\in D(\frakp,\chi)\subset D(\ggg,\chi)$.

Thus $D(\ggg,\chi)$ is not empty. The proof is completed.
\end{proof}

\subsection{} By Corollary \ref{coro: 5.2}, there  exists $\hhh\in D(\ggg,\chi)$ for any given $\chi\in \ggg_\bz^*$.
\begin{prop}\label{prop: 5.4} Let $\ggg=\ggg_\bz+\ggg_\bo$ be a completely solvable  Lie superalgebra, and $\chi\in \ggg_\bz^*$ given. Then there is an irreducible module $V$ of $\ggg$ such that $\dim V=p^{\frac{b_0^\chi}{2}}2^{\lfloor\frac{b_1^\chi}{2}\rfloor}$.
\end{prop}
\begin{proof} By Corollary \ref{coro: 5.2}, there  exists $\hhh\in D(\ggg,\chi)$.  If $\ggg$ coincides with $\hhh$, then all irreducible modules are one-dimensional. we are done.

In the following, we suppose that $\hhh$ is a proper subalgebra of $\ggg$. We prove the statement by induction on $\dim \ggg$.
By Lemma \ref{lem: 5.1}, $\hhh$ is contained in a subalgebra $\frakp$ of codimension one in $\ggg$, and $D(\frakp,\chi)\subset D(\ggg,\chi)$. By the inductive hypothesis along with Remark \ref{rem: 2.2}(1), there is an irreducible $\frakp$-module $W$ with
\begin{align}\label{eq: induction dim}
\dim W=\begin{cases} &p^{\frac{b_0^\chi}{2}-1}2^{\lfloor\frac{b_1^\chi}{2}\rfloor}\text{  if } \underline{\dim} \frakp=(\dim\ggg_\bz-1|\dim\ggg_\bo);\cr
&p^{\frac{b_0^\chi}{2}}2^{\lfloor\frac{b_1^\chi}{2}\rfloor-1} \text{ if } \underline{\dim} \frakp=(\dim\ggg_\bz|\dim\ggg_\bo-1).
\end{cases}
 \end{align}
 There certainly exists an irreducible module $V$ of $\ggg$ such that $V$ contains an irreducible $\frakp$-submodule isomorphic to $W$.
 We claim that
$\dim V=p^{\frac{b_0^\chi}{2}}2^{\lfloor\frac{b_1^\chi}{2}\rfloor}$.

Actually,  by Proposition \ref{prop: 4.5} and its proof, there is a subalgebra $\calh$ such that $\chi(\calh^{(1)})=0$ and the dimensions of $W$ and of $V$ are respectively formulated as  $\dim V=p^{\dim\ggg_\bz\slash \calh_\bz}2^{\dim \ggg_\bo\slash \calh_\bo}$, and  $\dim W=p^{\dim\frakp_\bz\slash \calh'_\bz}2^{\dim \frakp_\bo\slash \calh'_\bo}$,  where $\calh'_\bz=\ggg_\bz\cap\calh_\bz$, $\calh'_\bo=\ggg_\bo\cap\calh_\bo$. Note that $(\frakp_i+\calh_i\slash\calh_i)\cong
\frakp_i\slash \frakp_i\cap\calh_i$,  $\forall i\in\bbz_2$, as vector spaces.
From (\ref{eq: induction dim}), it yields that $\dim V=p^{\frac{b_0^\chi}{2}}2^{\lfloor\frac{b_1^\chi}{2}\rfloor}$.
\end{proof}

 \subsection{} From now on we turn to the situation of completely solvable restricted Lie superalgebras.

\begin{prop}\label{prop: 5.3} Let $\ggg$ be a completely solvable restricted Lie superalgebras with $p$-mapping $[p]$ on $\ggg_\bz$, and $\chi \in\ggg^*_\bz$ any given $p$-character. Then the following statements hold.
\begin{itemize}
\item[(1)] Each irreducible $U_\chi(\ggg)$-module $V$ is associated with certain subalgebra $\hhh$ with $\chi(\hhh^{(1)})=0$, such that $V$ has dimension $$p^{\dim\ggg_\bz\slash\hhh_\bz}2^{\dim\ggg_\bo\slash\hhh_\bo}$$
    and there is a one-dimensional $\hhh$-submodule in $V$.
    \item[(2)] All irreducible modules of $U_\chi(\ggg)$ have the same dimension.
\end{itemize}
\end{prop}

\begin{proof} (1) follows from Theorem \ref{lem: general solv}.

For (2),  we first notice that the even center $C(\ggg)_\bz$ is an ideal of $\ggg$,  and $C(\ggg)_\bz=C(\ggg)\cap \ggg_\bz\subset C(\ggg_\bz)$.
 Then we prove the proposition by  different steps.

(2.i) We claim that if $C(\ggg)_\bz^{[p]}=0$, then
every irreducible restricted representation of $\ggg$ is one-dimensional. Note that by definition $\ggg^{(1)}$ is nilpotent, acting nilpotently on $\ggg$. Hence there exists  a positive integer $k$ such that $X^{[p]^k}\in C(\ggg)_\bz$ for all $X\in \ggg_\bz$. The assumption that  $C(\ggg)_\bz^{[p]}=0$ entails that $\ggg_\bz^{(1)}\subset \text{rad}_p(\ggg_\bz)$ where $\text{rad}_p(\ggg_\bz)$ denotes $p$-radical which is the set of $X\in\ggg_\ev$ with $X^{[p]^n}=0$ for some positive integer $n\in \bbn$ (here $n$ is dependent on $X$). By the same arguments as in the proof \cite[Lemma 5.8.6(1)]{SF}, one has $\ggg_\bz=T\oplus \text{rad}_p(\ggg_\bz)$ where $T$ is a maximal torus of $\ggg_\bz$. In particular, $C(\ggg)_\bz\subset \text{rad}_p(\ggg_\bz)$, and  $\ggg^{(1)}\subset \text{rad}_p(\ggg_\bz)+\ggg_\bo$
    while  $[\ggg_\bo,\ggg_\bo]\subset \text{rad}_p(\ggg_\bz)$. Hence $\ggg^{(1)}\subset \text{rad}(\ggg)$.
For any irreducible restricted representation $\rho: \ggg\rightarrow \mathfrak{gl}(V)$,  $\rho(\text{rad}_p(\ggg_\bz))$ acts nilpotently on $V$, therefore acts trivially on $V$. So $\rho(\ggg^{(1)})$ acts trivially on $V$. By Lemma \ref{lem: moving}, $V$ is one-dimensional, which is of parity $0$\footnote{For representation categories of $U_\chi(\ggg)$, one can give parities arising from the given parity of homogeneous  generating space to \cite[Section 6]{CLW}.}. This $V$ is completely decided by a function $\lambda$  with $\lambda(\ggg^{(1)})=0$ and $\lambda(X^{[p]})=\lambda(X)^p$ for $X\in \ggg_\bz$.

(2.ii)  Suppose that $(W_i,\rho_i)$ ($i=1,2$) are two irreducible modules of $U_\chi(\ggg)$ with ${\rho_1}|_{C(\ggg)_\bz}={\rho_2}|_{C(\ggg)_\bz}$. Consider $\calh:=\Hom_\bk(W_1,W_2)$. Then $\calh=\calh_\bz+\calh_\bo$ becomes a $\ggg$-module by defining via homogenous elements $X\in \ggg_{|X|}$, and $\nu\in\calh_{|\nu|}$ where $|X|, |\nu|\in\bbz_2$ denote the parities of $X$ and $\nu$, respectively:
\begin{align}\label{eq: hom mod}
(X.\nu)(v)=\rho_2(X)\nu(v)-(-1)^{|X||\nu|}\nu(\rho_1(X)v).
\end{align}
Furthermore, it is readily shown that  $\calh$ is a
restricted $\ggg$-module, on which the representation is denoted by $\vartheta$. By definition, $\vartheta(\ggg)$ admits a $p$-mapping satisfying $\vartheta(C(\ggg)_\bz)^{[p]}=0$ because ${\rho_1}|_{C(\ggg)_\bz}={\rho_2}|_{C(\ggg)_\bz}$
  and $C(\ggg)_\bz\subset C(\ggg_\bz)$. Hence there is an irreducible submodule of $U_0(\ggg)$ in $\calh$, which  must be one-dimensional by (2.i), admitting parity $0$. We take such one, for example, $\bbs:=\bk\nu$. Then $\ggg^{(1)}$  acts trivially on $\nu$, and $\ggg$ acts on $\nu$ by a scalar $\lambda\in \ggg_\bz^*$ with $\lambda(\ggg^{(1)})=0$. Also  $\vartheta(X)\nu=\lambda(X)\nu$ for all $X\in \ggg$, consequently  $\lambda(X^{[p]})=\lambda(X)^p$ for $X\in\ggg_\bz$.  By (\ref{eq: hom mod}), $\rho_2(X)(\vartheta(w))-\vartheta(\rho_1(X)w)=\lambda(X)\vartheta(w)$ for  $X\in \ggg$  and $w\in W_1$. So the assignment $w\mapsto  \vartheta(w)\otimes 1$ for $w\in W_1$ defines a non-trivial $\ggg$-module homomorphism:   $W_1\rightarrow W_2\otimes_\bk \bk_{-\lambda}$, which is even. By Schur's Lemma, this homomorphism must be an isomorphism. Hence $\dim W_1=\dim W_2$.

 (2.iii) Note that $C(\ggg)_\bz\subset C(\ggg_\bz)$. So
 $\ggg_\bz$ possesses a $p$-mapping $[p]'$ with $C(\ggg)_\bz^{[p]'}=0$ (see \cite[Ch.2]{SF}), and for any $y\in\ggg_\bz$, $\xi(y):=y^{[p]}-y^{[p]'}$ belongs to $C(\ggg)_\bz$. So for any given irreducible representation $(\rho,W)$ of $\ggg$ on the space $W$, by Schur's Lemma $\xi(y)$ acts on $W$ by  scalar $c_W(y)^p$ where $c_W$ is a linear function on $C(\ggg)_\bz$. Then we have for any $y\in C(\ggg)_\bz$
 \begin{align}\label{eq: normal}
 \rho(y)^p= (\chi+ c_W)(y)^p\id_V.
 \end{align}
Note that $c_W$ can extend to a function on $\ggg_\bz$  which gives rise to a one-dimensional $U_{c_W}(\ggg)$-module, with respect to $(\ggg,[p]')$.
So for any two irreducible $U_\chi(\ggg)$-modules $W_1$ and $W_2$,  we have two new irreducible representations $\overline \rho_i$ on $\overline W_i:=W_i\otimes_\bk \bk_{-c_{W_i}}$ $(i=1,2)$, respectively. Then both of them  become $U_\chi(\ggg)$-modules with respect to $(\ggg,[p]')$, and  $\overline\rho_i|_{C(\ggg)_\bz}$ ($i=1,2$) coincide.  By (2.ii), $\dim \overline W_i$ ($i=1,2$) have the same dimension. Hence    both $W_1$  and $W_2$ have the same dimension.

 The proof is completed.
\end{proof}

\subsection{} Summing up Propositions \ref{prop: 5.3} and  \ref{prop: 5.4} we have
\begin{theorem}\label{thm: 5.5} Let $\ggg=\ggg_\bz+\ggg_\bo$ be a finite-dimensional {completely} solvable restricted Lie superalgebra over $\bk$. Then for any given $\chi\in \ggg^*_\bz$,  all irreducible $\ggg$-modules associated with $\chi$ have dimension
 $$p^{\frac{b^\chi_0}{2}}2^{{\lfloor\frac{b^\chi_1}{2}\rfloor}}.$$
\end{theorem}

The above theorem is an extension of  Kac-Weisfeiler's result on completely solvable restricted Lie algebras (see \cite[Theorem 1]{WK}). { As a corollary to Theorem \ref{thm: 5.5}, we have}
\begin{corollary} Conjecture \ref{conj} holds for completely solvable restricted Lie superalgebras.
\end{corollary}

    {
\section{Appendix: Minimal $p$-envelopes for a finite-dimensional Lie superalgebra over $\bk$}\label{sec: appendix}
Let $\ggg=\ggg_\bz\oplus \ggg_\bo$ be a finite-dimensional Lie superalgebra over $\bk$ of characteristic $p>2$. In this appendix section, we introduce the properties of  $p$-envelopes of $\ggg$. For more details on restricted Lie algebras, the reader may refer to \cite[\S2.5]{SF}.

 \subsection{Definition of $p$-envelopes}
 {
  \begin{defn} Keep the notations and assumptions as above.  A restricted Lie superalgebra $(\calg,[p])$  is said to be a $p$-envelope of $\ggg$ if there exists a homomorphism of Lie superalgebras $\bi:\ggg\rightarrow \calg$ such that $\bi$ is injective and the restricted Lie sub-superalgebra $\bi(\ggg)_p$ of $\calg$ generated by $\ggg$ coincides with $\calg$.
  \end{defn}


   In the following we make an example of $p$-envelopes which is taken from the universal enveloping algebra $U(\ggg)$. We first recall that an associative superalgebra $\mathfrak{A}=\mathfrak{A}_\bz\oplus \mathfrak{A}_\bo$ can be endowed with structure of a restricted Lie superalgebra, which is denoted by $\mathfrak{A}^-$, where  the underline space of $\mathfrak{A}^-$  is   $\mathfrak{A}$ itself, and the Lie bracket is defined via $[u_1,u_2]:=u_1u_2-(-1)^{|u_1||u_2|}u_2u_1$ for $\bbz_2$-homogeneous elements  $u_i\in \mathfrak{A}_{|u_i|}$, $|u_i|\in \bbz_2$, $i=1,2$. And the $p$-mapping of $\mathfrak{A}_\bz$ is just the usual $p$th power in $\mathfrak{A}_\bz$.

  \begin{example}

  %
  Let $\bi:\ggg\rightarrow U(\ggg)$ be the canonical imbedding of $\ggg$ into $U(\ggg)$. Recall that in $U(\ggg_\bz)^-\subset U(\ggg)^-$, the Lie subalgebra $\bi(\ggg_\bz)$ generates a restricted Lie subalgebra $\bi(\ggg_\bz)_p=\bigoplus_{i=0}^\infty \bi(\ggg)^{p^i}$ of $U(\ggg_\bz)^-$, where $\bi(\ggg)^{p^i}=\{\bi(g)^{p^i}\mid g\in\ggg_\bz\}$ (see \cite{Zass} or \cite[\S5.2]{SF}).  Take
  $$\calg=\bi(\ggg_\bz)_p\oplus \bi(\ggg_\bo)\subset U(\ggg)^-.$$
Then $\calg$ becomes a $p$-envelope of $\ggg$.
\end{example}

 Such a $p$-envelope as above is usually called the universal $p$-envelope of $\ggg$, which we denote by $\widehat\ggg$.
%
  By straightforward calculations, it is not hard to see that $\bi(\ggg)$ is an ideal of $\widehat\ggg$. Consider the superalgebra $\text{SDer}(\ggg)$ of super deriviations on $\ggg$, and  the homomorphism $\ad: \widehat\ggg\rightarrow \text{SDer}(\ggg)$ defined via sending  $x\mapsto \ad x|_{\bi(\ggg)}$. Then  $\ker(\ad)$ coincides with the center $C(\widehat\ggg)$ of $\widehat\ggg$.

}

 \subsection{} By the same arguments as in the Lie algebras case (see \cite[\S2.5]{SF}, we have the following basic results.
 \begin{lemma}\label{lem: app} The following statements hold.
 \begin{itemize}
 \item[(1)] There is a finite-dimensional $p$-envelope $\calg$ of $\ggg$ such that $\calg=(\ggg_\bz)_p\oplus \ggg_\bo$ where $(\ggg_\bz)_p$ is a $p$-envelope of $\ggg_\bz$.
     \item[(2)] There is a minimal finite-dimensional $p$-envelope $\calg$ of $\ggg$ satisfying (1).
 \item[(3)] Any two minimal dimensional $p$-envelopes of $\ggg$ are isomorphic, as Lie superalgebras.
     \end{itemize}
 \end{lemma}

\begin{proof} For the part (1), we choose a sub-superspace $V$ in the center $C(\widehat\ggg)$ of $\widehat\ggg$ such that $C(\widehat\ggg)=V\oplus (C(\widehat\ggg)\cap \phi(\ggg))$. This $V$ naturally becomes  an ideal of $\hat\ggg$.  Consider $\widetilde\ggg:=\widehat\ggg\slash V$. It is easily seen that $\widetilde\ggg$ is endowed with structure of restricted Lie superalgebras arising from the one of $\widehat\ggg$.
Furthermore, $\dim \widetilde\ggg=\dim \widehat\ggg\slash C(\widehat\ggg)+\dim C(\widehat\ggg\cap \phi(\ggg)$. Note that $\phi(\ggg)$ is an ideal of $\widehat\ggg$, and the homomorphism $\ad: \widehat\ggg\rightarrow \text{SDer}(\ggg)$ defined via sending  $X\mapsto \ad X|_{\phi(\ggg)}$,  admits $\ker(\ad)=C(\widehat\ggg)$. Hence $\dim\widehat\ggg\slash C(\widehat\ggg)\leq \dim \text{SDer}(\ggg)<\infty$. Hence $\dim\widetilde\ggg<\infty$.

By the choice of $V$, there is an embedding $\psi$ of $\ggg$ into $\widetilde\ggg$, i.e. $\psi=\textsf{p}\circ\phi$ for the natural surjective homomorphism $\textsf{p}:\widehat\ggg\rightarrow \widetilde\ggg$. Consequently, it is readily known that $\widetilde\ggg$ is a $p$-envelope of $\ggg$. Such $\widetilde\ggg$ satisfies the requirement that $\widetilde\ggg_\bz=\psi(\ggg_\bz)_p$ is a $p$-envelope of $\ggg_\bz$ and $\widetilde\ggg=\psi(\ggg_\bz)_p\oplus \psi(\ggg_\bo)$.

As to (2), note that in the above arguments, the center of $\widetilde\ggg$ lies in $\psi(\ggg)$. By an analogue of the arguments  in the proof of \cite[Theorem 2.5.8]{SF}, it can be proved that $\widetilde\ggg$ is a minimal finite-dimensional $p$-envelope of $\ggg$ satisfying (1).

 The proof for (3) is also an analogue of that of \cite[Theorem 2.5.8]{SF}. We omit the details.
\end{proof}

\subsection*{Acknowledgement} The author expresses his sincere thanks to the anonymous referee for his/her helpful comments and suggestions, and to Dr. Priyanshu  Chakraborty  for assistance in English expression.

\end{document}